\documentclass[11pt]{amsart}

\usepackage[english]{babel}
\usepackage[utf8]{inputenc}
\usepackage[T1]{fontenc}

\usepackage{amsmath}
\usepackage{amssymb}
\usepackage{amsfonts}
\usepackage{amssymb}
\usepackage{amsthm}
\usepackage{amscd}

\usepackage[hidelinks]{hyperref}
\usepackage{geometry}

\newtheorem{definition}{Definition}[section]
\newtheorem{lemma}[definition]{Lemma}
\newtheorem{theorem}[definition]{Theorem}
\newtheorem{proposition}[definition]{Proposition}

\newcommand{\setsymbol}[1]{\ensuremath{\mathbb{#1}}}%
\newcommand{\R}{\setsymbol{R}}%
\newcommand{\Sphere}{\setsymbol{S}}%

\DeclareMathOperator{\divA}{div}
\DeclareMathOperator{\trace}{tr}
\DeclareMathOperator{\seg}{seg}
\DeclareMathOperator{\id}{id}

\DeclareMathOperator{\dist}{dist}
\DeclareMathOperator{\sn}{sn}
\DeclareMathOperator{\cs}{cs}
\DeclareMathOperator{\Ric}{Ric}

\author{Florian Johne}
\address{Mathematisches Institut, Universität Freiburg, Ernst-Zermelo Str. 1, 79104 Freiburg, Germany}
\email{florian.johne@math.uni-freiburg.de}

\author{Lauro Silini}
\address{Institute of Science and Technology Austria (ISTA), Am Campus 1, 3400 Klosterneuburg,
Austria}
\email{Lauro.Silini@ist.ac.at}

\title{Alexandrov Theorem for constant weighted mean curvature surfaces}

\begin{document}

\begin{abstract}
 We prove a Heintze--Karcher type inequality for a large class of log-convex weights
 in Euclidean and Hyperbolic spaces. As a consequence we obtain an Alexandrov theorem for
 $\lambda$-self expanders.
\end{abstract}

 \maketitle

\section{Introduction}

Among the hypersurfaces of Euclidean space the flat plane and the round sphere
are the simplest examples. In particular, they both have constant mean curvature
and are totally umbilic. It is classical that complete totally umbilic hypersurfaces
are either flat planes or round spheres. A more intriguing question is whether 
closed hypersurfaces with constant mean curvature (CMC) are round spheres?

A first result in this direction was proved in 1853 by J.-H.~Jellett \cite{Jellett-1853-Classification-CMC-Sharshaped}, who used integral formulas (later named after H.~Minkowski)
to prove that any star-shaped CMC surface is a round sphere. 
In 1900 H.~Liebmann \cite{Liebmann-1900-Classification-CMC-Convex} proved that surfaces in $\R^3$ with constant mean curvature
and positive Gauß curvature (or surfaces with constant positive Gauß curvature)
are round spheres.

The classification problem for CMC surfaces received renewed interest in 1951, when H.~Hopf \cite{Hopf-1951-HopfDifferential} used a quadratic holomorphic differential to prove that any CMC immersion of genus zero into $\R^3$ is a round sphere, and asked in his 1956 lectures on \emph{Differential Geometry in the Large} whether the round sphere is the unique CMC immersed surface \cite[p.~131]{Hopf-LectureNotes-DifferentialGeometryInTheLarge}. The ingenious invention of the Moving Plane Method allowed A.~D.~Aleksandrov in 1956 \cite{Aleksandrov-1956-CMC-Classification} to prove that any embedded CMC hypersurface in $\R^{n+1}$ is a round sphere.

It came as a surprise when W.-Y.~Hsiang \cite{Hsiang-1982-CMC-Examples-HigherDimension} and W.-Y~Hsiang, Z.-H.~Teng
and W.-C.~Yu \cite{Hsiang-Ten-Yu-1982-CMC-Examples-HigherDimensions} constructed immersed CMC surfaces different from the round sphere
in $\R^n$ for $n \geq 4$.
Finally, in 1986 H.~Wente \cite{Wente-1986-WenteTorus} constructed a CMC torus in $\R^3$, therefore giving a negative answer to the above question by H.~Hopf.

Moreover, it is a folklore result that J.~Simons' identity \cite{Simons-1968-MinimalVarieties}
implies that that convex CMC surfaces are round spheres. A different proof 
of Alexandrov's Theorem was given in 1977 by R.~Reilly \cite{Reilly-1977-IntegralFormula} using the integral formula bearing his name. A.~Ros \cite{Ros-1987-Reilly, Ros-1988-Reilly} used the Reilly formula to obtain a Heintze--Karcher type inequality, in conjunction with the Hsiang--Minkowski formulas this implies a Alexandrov Theorem for the higher order mean curvatures.

The classification of embedded CMC hypersurfaces is also intriguing in other ambient spaces:
While there is an abundance of examples of embedded CMC hypersurfaces in the sphere,
there is an analogue of Alexandrovs Theorem for the hemi-sphere and for hyperbolic 
space. The proof by the method of moving planes by A.~D.~Aleksandrov \cite{Aleksandrov-1956-CMC-Classification} extends to this setting, however  S.~Montiel and A.~Ros \cite{Montiel-Ros-1991-Alexandrov} provided an alternative proof combining the Minkowski formula
with a variant of the Heintze--Karcher inequality (with a different proof
compared to earlier work by A.~Ros). In a milestone S.~Brendle \cite{Brendle-2013-Alexandrov-WarpedProducts}
extended the classification of embedded CMC surfaces to a large class of warped products,
in particular the Schwarzschild manifold, by combining a weighted Minkowski formula
with a weighted Heintze--Karcher inequality, which in turn was proved by deforming
the hypersurfaces by parallel surfaces with respect to a conformal metric.
Heintze--Karcher inequalities in substatic manifolds have also been proven by J.~Li and C.~Xia \cite{Li-Xia-2019-IntegralFormula-SubstaticManifolds}, and
by M.~Fogagnolo and A.~Pinamonti \cite{Fogagnolo-Pinamonti-2022-IntegralEstimates-SubstaticManifolds}, see also their
joint work with S.~Borghini \cite{Borghini-Fogagnolo-Pinamonti}
on the equality case.

In recent years variational problems with respect to weighted volume and weighted area
have received increased interest. Given a smooth weight $\rho: \R^{n+1} \rightarrow (0,\infty)$ we consider the weighted area $A_{\rho}(\Sigma) = \int_{\Sigma} \rho \; d\sigma$.
The associated first variation is given by
\begin{equation}
 \left. \frac{d}{dt} A_{\rho}(\Sigma_t) \right|_{t=0} 
 = \int_{\Sigma} f (H + \langle D \log \rho, \nu \rangle) \;  \rho \; d\sigma.
\end{equation}
and the associated second variation on a critical point is given by
\begin{equation}
\begin{aligned}
    &\left. \frac{d^2}{dt^2} A_{\rho}(\Sigma_t) \right|_{t=0}  \\
   &=
     \int_{\Sigma}
   \rho \left( -f \Delta_{\Sigma} f - 
   \left( |h_{\Sigma}|^2 + \Ric_N(\nu, \nu) \right) f^2
   + f^2 (D_N^2 \log \rho)(\nu, \nu)
   - f
   \langle D_{\Sigma} \log \rho, D_{\Sigma} f \rangle
   \right) 
   \rho \, d\sigma,
   \end{aligned}
\end{equation}
where $(\Sigma_t)_{t \in (-\epsilon, \epsilon)}$ is a variation of $\Sigma$
with normal velocity $f \in C^{\infty}_c(\Sigma)$,
compare \cite{Brendle-Hirsch-J-2024-GeneralizedGeroch}.
This motivates the definition of the weighted mean curvature given by
$H_{\rho} = H + \langle D \log \rho, \nu \rangle$.
Surfaces of constant weighted mean curvature (CWMC surfaces) satisfy $H_{\rho} = \lambda$ for some $\lambda \in \R$.

An example are self-similarly shrinking solutions to mean curvature flow:
These hypersurfaces are characterized by the equation $H = \frac{1}{2} \langle x, \nu \rangle$,
one may view these objects as weighted minimal surfaces with respect to
the weighted area $A_S(\Sigma) = \int_{\Sigma} \exp\left(- \alpha |x|^2 \right) \; d\sigma$,
where $\alpha > 0$.
Results by G.~Huisken \cite{Huisken-1990-MonotonicityFormula}, A.~Stone \cite{Stone-1994-MCF-Singularities},
and T.~Colding and W.~Minicozzi \cite{Colding-Minicozzi-2012-GenericMCF} yield a 
classification result: complete mean convex self-similarly
shrinking solutions to mean curvature flow with polynomial volume growth
are given by cylinders $\Sphere^{n-k} \times \R^k \subseteq \R^{n+1}$ for $0 \leq k \leq n-1$; these proofs rely on Simons' identity.

Another important class of examples are radially symmetric and log-convex densities.
For such a density centered spheres are stable critical points for the weighted area.
In fact, among radially symmetric densities centered spheres are stable
if and only if the weight is log-convex.
In an important result G.~Chambers \cite{Chambers-2019-LogConvexDensityConjecture}
showed that the solutions of the isoperimetric problem
for radially symmetric and log-convex densities
are given by balls centered at the orgin.

The most important example of a radially symmetric log-convex weight
is the expander weight $\rho(x) = \exp(\alpha |x|^2)$ for $\alpha > 0$.
The weighted mean curvature associated to the expander weight is given by $H_{\rho} = H + 2 \alpha \langle x, \nu \rangle$. In the literature there are several conventions for the signs
in the previous equation. Hypersurfaces with zero weighted mean curvature are
the self-expanders to mean curvature flow. 
Examples of hypersurfaces with constant positive weighted mean curvature are given
by spheres centered at the origin, and planes through the origin.

Previously S.~Ancari and X.~Cheng \cite{Ancari-Cheng-2023-Properties-SelfExpanders}
studied the rigidity properties of $\lambda$-self expanders:
among other results they showed a rigidity theorem under a pinching condition
on the second fundamental form \cite[Theorem 1.1]{Ancari-Cheng-2023-Properties-SelfExpanders} and a rigidity theorem for convex $\lambda$-self expanders \cite[Theorem 1.3]{Ancari-Cheng-2023-Properties-SelfExpanders}.

In this work we establish an Alexandrov theorem for embedded $\lambda$-self expanders.

\begin{theorem}[Alexandrov Theorem for expander weight] \ \\
 Let $n \geq 2$ and $\alpha > 0$. Suppose $\Sigma^n$ is a closed, connected and embedded submanifold in expander space $(\R^{n+1} , \delta, \exp(\alpha |x|^2) \; d\lambda)$ with
 constant weighted mean curvature $H_{\rho} \neq 0$.
 Then $\Sigma$ is a round sphere centered at the origin.
 \label{Theorem:AlexandrovTheorem-ExpanderWeight}
\end{theorem}

We remark that the situation is very different in the self-shrinking case, i.e. for the weight $\rho(x) = \exp(-\alpha |x|^2)$ for $\alpha > 0$. In that case
the problem is more flexible, and Q.-M.~Cheng and G.~Wei \cite{Cheng-Wei-2021-Lambda-Shrinkers-Examples} first
constructed for $n \geq 2$ embedded examples with topology $\Sphere^1 \times \Sphere^{n-1}$,
and later together with J.~Lai \cite{Cheng-Lai-Wei-2024-Lambda-Shrinkers-Examples} embedded convex non-round examples with the topology of the sphere.

Our proof follows a classical idea for the proof of Alexandrov's Theorem:
We combine a weighted Minkowski formula (Lemma \ref{Lemma:WeightedMinkowskiFormula}) with a weighted Heintze--Karcher inequality in expander space.
For the proof of the weighted Heintze--Karcher inequality (Theorem \ref{Theorem:HeintzeKarcherInequality}) we use parallel surfaces
with respect to a conformal metric, an idea 
which appeared in work of S.~Brendle on Alexandrov's Theorem 
in warped product spaces \cite[Section 3]{Brendle-2013-Alexandrov-WarpedProducts}.
A crucial ingredient in the proof of the weighted Heintze--Karcher inequality is a generalization of the classical trace inequality
relating the norm of the second fundamental form and the mean curvature (cf.\ Lemma \ref{Lemma:TraceInequality}); this new inequality detects 
centered spheres.

We remark that different weighted Heintze--Karcher inequalities
have previously been proved by F.~Morgan \cite{Morgan-2005-Densities} and by 
M.~Batista, M.~P.~Cavalcanta and J.~Pyo \cite{Batista-Cavalcanta-Pyo-2014-WeightedHeintzeKarcher}. 

In fact, the above method of proof generalizes to a much wider setting:
Euclidean space $(\R^{n+1}, \delta)$ may be replaced by hyperbolic space with nonpositive curvature $\kappa \leq 0$; and log-quadratic weight might be replaced
by a class of weights satisfying a differential inequality.

\begin{theorem}[Alexandrov Theorem in nonpositive curvature with weights] \ \\
 Let $n \geq 2$, $\kappa \leq 0$ and let $\rho: M^{n+1}_{\kappa} \rightarrow (0,\infty)$
 be a nonconstant positive radially symmetric weight (with respect to the origin $o \in M^{n+1}_{\kappa}$), i.e.\ there exists $\psi: (0,\infty) \rightarrow \R$, such that $\rho(x) = \exp(\psi(r(x)))$. Moreover, assume that $\psi$ additionally satisfies
 for all $r >0$ the inequalities $\psi'(r) \geq 0$ and 
 \begin{equation}
 A_{\kappa}(r) =  \frac{d}{dr} \left( \psi''(r) + \frac{(\psi'(r))^2}{2} + (n+1) \frac{\cs_{\kappa}(r) \psi'(r)}{\sn_{\kappa}(r)} + \kappa n \psi(r) 
  \right) \geq 0.
  \label{equation:Monotonicity}
 \end{equation}
Suppose $\Sigma^n$ is a closed, connected and embedded submanifold
in the weighted manifold $(M^{n+1}_{\kappa}, \rho \, dx)$ with 
constant wegithed mean curvature $H_{\rho} \neq 0$.
Then $\Sigma$ is a geodesic sphere centered at the origin $o \in M^{n+1}_{\kappa}$.
\label{Theorem:AlexandrovTheorem-GeneralWeight}
\end{theorem}

Observe that for $\kappa = 0$ the condition $A_0 \geq 0$ is satisfied
for weights such as $\psi(r) = r^m$ and $\psi(r) = \log(1 + \exp(r^m))$
for $m \geq 2$, $\psi(r) = \cosh(r)$, however the condition fails for 
asymptotically linear weights such as $\psi(r) = \sqrt{1+r^2}$.
For $\kappa = -1$ the condition $A_{-1} \geq 0$ is satisfied
for weights such as $\psi(r) = \cosh(r)$, but again fails
for asymptotically linear weights such as $\psi(r) = \log(\cosh(r))$.
In Appendix B we show that in the euclidean setting the condition
$A_0 \geq 0$ implies log-convexity of the weight.

\vspace{2mm}

It is an interesting conjecture whether the conclusion of Theorem \ref{Theorem:AlexandrovTheorem-GeneralWeight} holds for all radially symmetric and log-convex weights.
We remark that the analogue of the isoperimetric problem with log-convex weights
(resolved by G.~Chambers \cite{Chambers-2019-LogConvexDensityConjecture}) was studied by the second author \cite{Silini-2024-HyperbolicLogConvexDensity} with applications
to symmetric solutions of the isoperimetric problem in compact rank-one symmetric spaces.

 This article is structured as follows:
 In Section 2 we recall our setup and preliminaries on weighted manifolds. In Section 3 we prove the weighted Minkowski formula for the expander weight, and the expander inequality
 relating the second fundamental form and the weighted mean curvature,
 and a hypersurface identity. In Section 4 we prove the weighted Heintze--Karcher type inequality and our main theorem. In the appendix we outline
 the generalization to more general weights and ambient spaces with constant negative curvature
 as stated in Theorem \ref{Theorem:AlexandrovTheorem-GeneralWeight}.
 
 \textbf{Addendum:} \ \\
 We note that J.~Bai and C.~Xia \cite{Bai-Xia-2026-AlexandrovTypeTheorem} have simultaneously and independently obtained similar results via a related technique.

 \textbf{Acknowledgements:} \ \\
The authors would like to thank the organisers of the 
``The Pisan workshop saga: Ep.V - Geometric analysis strikes back'',
where this project started. We also express our gratitude to Daniele Semola
for interesting discussions on the matter.
  
\section{Preliminaries}

For $\kappa \leq 0$ we recall the standard functions $\sn_{\kappa}, \cs_{\kappa}: (0,\infty) \rightarrow (0,\infty)$ given by
\begin{align*}
 \sn_{\kappa}(r)
 =
 \begin{cases}
  \frac{\sinh(\sqrt{-\kappa} r)}{\sqrt{-\kappa}} & \kappa < 0, \\
  r & \kappa = 0,
 \end{cases}
 \; \text{and} \;
 \cs_{\kappa}(r)
 =
 \begin{cases}
  \cosh(\sqrt{-\kappa} r) & \kappa < 0, \\
  1 & \kappa = 0.
 \end{cases}
\end{align*}
Note that for all $\kappa \leq 0$ the identities
$\sn_{\kappa}' = \cs_{\kappa}$ and $\cs_{\kappa}' = - \kappa \sn_{\kappa}$.

For $\kappa \leq 0$ we denote by $(M^{n+1}_{\kappa}, \delta_k)$ the simply--connected space form
with constant curvature $\kappa$, which is equipped with the volume measure $dx$
and its Levi--Civita connection $D$. More concretely, we use the polar coordinate model
on $(0,\infty) \times \Sphere^n$ with metric given by $\delta_k = dr \otimes dr + \sn_{\kappa}^2(r) \, \delta_{\Sphere^n}$.
For a base point $o \in M^{n+1}_{\kappa}$ we define the distance function
$r: M^{n+1}_{\kappa} \to [0,\infty)$ by $r(x) = \dist_{\delta_k}(x,o)$ and the radial vector field
$\partial_r = Dr(x)$.

For a two-sided hypersurface $\Sigma$ in $(M^{n+1}_{\kappa}, \delta_k)$ we denote its
unit normal vector field by $\nu$, its area measure by $d\sigma$ and its induced
Levi--Civita connection by $D^{\Sigma}$. Then the ambient Levi--Civita connection
splits as $D_X Y = D^{\Sigma}_X Y + A(X,Y)$, where $A$ denotes the 
(vector-valued) second fundamental form of $\Sigma$. We denote the (scalar-valued)
second fundamental form of $\Sigma$ by $h(X,Y) = \delta_k(D_X \nu, Y)$
and the mean curvauture by $H = \trace_{\delta_k} h$.

Throughout the article $\rho: M^{n+1}_{\kappa} \rightarrow (0,\infty)$
denotes a smooth positive weight. This induces a weighted volume measure
$dx_{\rho} = \rho \, dx$ and a weighted area measure $d\sigma_{\rho}
= \rho \; d\sigma$. Then, for an open set $\Omega^{n+1}$ in $M^{n+1}_{\kappa}$
its weighted volume is given by 
$V_{\kappa, \rho}(\Omega) = \int_{\Omega} dx_{\rho}$,
and for a hypersurface $\Sigma^n$ in $M^{n+1}_{\kappa}$ its weighted area is given by
$A_{\kappa, \rho}(\Sigma) = \int_{\Sigma} d\sigma_{\rho}$.
By an elementary computation one checks that the first variation of weighted area is given by
\begin{equation}
 \left. \frac{d}{dt} A_{\rho}(\Sigma_t) \right|_{t=0} 
 = \int_{\Sigma} f (H + \langle D \log \rho, \nu \rangle) \;  \rho \; d\sigma,
\end{equation}
where $f \in C^{\infty}_c(\Sigma)$ denotes the normal speed.
This motivates the definition of the weighted mean curvature given by
$H_{\rho} = H + \langle D \log \rho, \nu \rangle$.
Surfaces of constant weighted mean curvature (CWMC surfaces)
satisfy $H_{\rho} = \lambda$ for some $\lambda \in \R$.

Suppose $X$ is a vector field on $M^{n+1}_{\kappa}$. Then its weighted divergence
is given by
\[
 \divA_{\rho} X
 = \rho^{-1} \divA (\rho X) = \divA X + \langle D \log \rho, X \rangle.
\]
In particular, for a smooth function $f$ on $M^{n+1}_{\kappa}$ we obtain the weighted Laplacian
\[
 \Delta_{\rho} f = \rho^{-1} \divA(\rho Df) = \Delta f + \langle D \log \rho, Df \rangle.
\]

Similarly, on a hypersurface $\Sigma^n$ in $M^{n+1}_{\kappa}$ we define
for a tangential vector field $V$ on $\Sigma$ and a smooth function $f$ on $\Sigma$
the weighted divergence and the weighted Laplacian by
\begin{align*}
 \divA_{\rho}^{\Sigma} V
 &= \rho^{-1} \divA^{\Sigma} (\rho V)
 = \divA^{\Sigma} V + \langle D^{\Sigma} \log \rho, V \rangle, \\
 \Delta_{\rho}^{\Sigma} f 
 &= \rho^{-1} \divA^{\Sigma} (\rho D^{\Sigma} f)
 = \Delta^{\Sigma} f + \langle D^{\Sigma} \log \rho, D^{\Sigma} f \rangle,
\end{align*}
where $\divA^{\Sigma}$ and $\Delta^{\Sigma}$ denote
the divergence and the Laplace operator on the hypersurface $\Sigma$.

\section{Geometric observations for expander weight}

We first prove a weighted version of the Minkowski formula
for hypersurfaces:
 \begin{lemma}[Weighted Minkowski identity for expander space] \ \\ 
 Let $n \geq 2$ and $\alpha \in \R$. Suppose $\Sigma$ is a closed and embedded submanifold
 in expander space $(\R^{n+1}, \delta, \, \exp(\alpha |x|^2)) \, d\lambda)$.
 Then we have the weighted Minkowski identity 
  \begin{equation}
   \int_{\Sigma}
   H_{\rho} \langle x, \nu \rangle
   \; d\sigma_{\rho}
   =
   \int_{\Sigma}
   (n + 2 \alpha |x|^2) \; d \sigma_{\rho},
  \end{equation}
  where $d\sigma_{\rho} = \exp(\alpha |x|^2) \, d\sigma$ denotes the weighted area measure.
  \label{Lemma:WeightedMinkowskiFormula}
 \end{lemma}
 
 \begin{proof}
 By direct computation one obtains the well-known identity
 \begin{align*}
  \frac{1}{2} \Delta^{\Sigma} |x|^2 = n - H \langle x, \nu \rangle.
 \end{align*}
 Then we obtain from the definition of the weighted Laplacian on $\Sigma$ that
 \begin{align*}
  \frac{1}{2} \Delta^{\Sigma}_{\rho} |x|^2
  &= \frac{1}{2} \Delta^{\Sigma} |x|^2 + \left \langle D^{\Sigma} (\alpha |x|^2), D^{\Sigma}
  \left( \frac{1}{2} |x|^2 \right) \right \rangle \\
  &= n - H \langle x, \nu \rangle + \frac{1}{2} \alpha |D^{\Sigma} |x|^2|^2 \\
  &= n - H \langle x, \nu \rangle + 2 \alpha |x^{\parallel}|^2 \\
  &= n + 2 \alpha |x|^2 - 2 \alpha \langle x,\nu \rangle^2 - H \langle x, \nu \rangle \\
  &= n + 2 \alpha |x|^2 - H_{\rho} \langle x, \nu \rangle.
 \end{align*}
 We integrate this identity and apply the divergence theorem
 to conclude the proof.
 \end{proof}

Next, we prove an identity which reduces for $\alpha = 0$
to the trace identity $n|h|^2 - H^2 = n |\overset{\circ}{h}|^2$ for the second fundamental
form in Euclidean space. Note that for the shrinker weight, i.e.\ $\alpha < 0$,
we do not obtain an inequality with the correct sign.

\begin{lemma}[Trace inequality for second fundamental form for expander weight] \ \\
Let $n \geq 2$ and $\alpha \in \R$. Suppose $F: M^n \rightarrow \R^{n+1}$
is an immersion of the closed manifold $M^n$
into expander space $(\R^{n+1}, \delta, \exp(\alpha |x|^2) \, d\lambda)$. Then we have the pointwise identity
\begin{align*}
  (|h|^2 + 2 \alpha)(2\alpha |x|^2 + n) - H_{\rho}^2
  =
     (|h|^2 + 2\alpha) 2 \alpha |x^{\parallel}|^2 
    + |\overset{\circ}{h}|^2 (n + 2 \alpha \langle x, \nu \rangle^2)
    + \frac{2\alpha}{n}
    \left(
     H \langle x, \nu \rangle - n
    \right)^2.
 \end{align*}
 Moreover, for $\alpha \geq 0$
 the right-hand side is nonnegative; and for $\alpha > 0$ it vanishes if and only
 if $\Sigma^n = F(M^n)$ is a round sphere centered at the origin.
 \label{Lemma:TraceInequality}
\end{lemma}
 
 \begin{proof} 
  We compute by using the definition of the weighted
  mean curvature and the trace decomposition of symmetric
  two tensors that 
  \begin{align*}
   &(|h|^2 + 2 \alpha)(2\alpha |x|^2 + n) - H_{\rho}^2 \\
    =&
    (|h|^2 + 2 \alpha) 2 \alpha |x^{\parallel}|^2 
    + (|h|^2 + 2 \alpha) 2 \alpha \langle x, \nu \rangle^2
    + n |h|^2 + 2 \alpha n
    -
    (H^2 + 4 \alpha H \langle x, \nu \rangle + 4 \alpha^2 \langle x, \nu \rangle^2) \\
    =&
    (|h|^2 + 2\alpha) 2 \alpha |x^{\parallel}|^2 
    + |\overset{\circ}{h}|^2 (n + 2 \alpha \langle x, \nu \rangle^2)
    + 2 \alpha n  + \frac{2\alpha H^2}{n} \langle x, \nu \rangle^2
    - 4 \alpha H \langle x, \nu \rangle \\
        =&
    (|h|^2 + 2\alpha) 2 \alpha |x^{\parallel}|^2 
    + |\overset{\circ}{h}|^2 (n + 2 \alpha \langle x, \nu \rangle^2)
    + \frac{2\alpha}{n}
    \left(
     H \langle x, \nu \rangle - n
    \right)^2.
  \end{align*}
  For $\alpha \geq 0$ the right-hand side is nonnegative as a sum of three nonnegative terms. Moreover, for $\alpha > 0$
  vanishing of the right-hand side forces that $\Sigma$ is totally umbilic,
  and hence a sphere; then $x^{\parallel} = 0$ forces that the sphere is centered.
 \end{proof}

 \begin{lemma} \ \\
  Let $n \geq 2$ and $\alpha > 0$. Suppose $F: M^n \rightarrow \R^{n+1}$
  is an immersion of the manifold $M^n$ into expander space $(\R^{n+1}, \delta, \exp(\alpha |x|^2) \, d\lambda)$. Then the function $\phi: \R^{n+1} \rightarrow \R$ given by 
  $\phi(x) = n + 2 \alpha |x|^2$ satisfies the equation
  \begin{equation*}
   \Delta^{\Sigma}_{\rho} \phi = 4 \alpha (\phi - H_{\rho} \langle x, \nu \rangle).
  \end{equation*}
  \label{Lemma:HypersurfaceIdentity}
 \end{lemma}

 \begin{proof} 
  We compute
  \begin{align*}
   \Delta^{\Sigma}_{\rho} \phi
   &= \Delta^{\Sigma} \phi + \langle D^{\Sigma} \log \rho, D^{\Sigma} \phi \rangle \\
   &= 2 \alpha (\Delta^{\Sigma} |x|^2 + \alpha |D^{\Sigma} |x|^2|^2) \\
    &= 2 \alpha ( 2n - 2 H \langle x, \nu \rangle + 4 \alpha |x^{\parallel}|^2 ) \\
   &= 2 \alpha ( 2n - 2 H \langle x, \nu \rangle + 4 \alpha (|x|^2 - \langle x, \nu \rangle^2) ) \\
   &= 4\alpha (n + 2 \alpha |x|^2 - \langle x, \nu \rangle (H + 2 \alpha \langle x, \nu \rangle))  \\
   &= 4 \alpha (\phi - H_{\rho} \langle x, \nu \rangle),
  \end{align*}
  where we used the definition of the weighted Laplacian on $\Sigma$,
  the definition of the expander weight, the identity relating the ambient Laplacian
  and the hypersurface Laplacian.
 \end{proof}

\section{Weighted Heintze--Karcher inequality} 

In this section we prove the following Heintze--Karcher inequality, which is central
for our argument. Our argument uses
parallel surfaces with respect to a conformal metric;
this closely follows work of work of S.~Brendle \cite[Section 3]{Brendle-2013-Alexandrov-WarpedProducts}.

 \begin{theorem}[Weighted Heintze--Karcher inequality
 for expander weight] \ \\ 
 Let $n \geq 2$ and $\alpha > 0$. Suppose $\Sigma^n$ is a closed, connected and embedded submanifold
 with positive weighted mean curvature $H_{\rho} >0$ bounding a non-empty open and connected subset $\Omega$ in $\R^{n+1}$. Then we have the weighted Heintze--Karcher inequality
  \begin{equation}
   \int_{\Omega}
   (n+1 + 2\alpha |x|^2) \; d\lambda_{\rho}
   \leq 
   \int_{\Sigma}
   \frac{n + 2\alpha |x|^2}{H_{\rho}} \; d\sigma_{\rho},
  \end{equation}
  where $d\lambda_{\rho} = \exp(\alpha|x|^2) \, d\lambda$
  and $d\sigma_{\rho} = \exp(\alpha |x|^2)|_{\Sigma} \, d\sigma$
  are the weighted volume and area measures.
  Moreover, equality holds if and only if $\Sigma^n$ is a sphere centered at the origin. 
  \label{Theorem:HeintzeKarcherInequality}
 \end{theorem}

 Suppose $\Sigma^n$ is a connected, closed and embedded submanifold of $(\R^{n+1}, \delta)$.
 The Jordan--Brower separation theorem implies that $\Sigma^n$ is orientable.
 Moreover, there exists an open subset $\Omega$ of $\R^{n+1}$, such that $\partial \Omega = \Sigma$.
 We choose the outward point unit normal vector field $\nu \in \Gamma(N \Sigma)$,
 and assume that $\Sigma$ has positive weighted mean curvature $H_{\rho} = 
 H + 2 \alpha \langle x, \nu \rangle$ with respect to this choice of unit normal.

 For $\phi: \R^{n+1} \rightarrow [n, \infty)$ given by $\phi(x) = n + 2 \alpha |x|^2$
 we consider the conformal metric $g = \phi^{-2} \delta$. We define $u: \overline{\Omega} \rightarrow [0,\infty)$ by $u(p) = d_{g}(p, \Sigma)$. 
 We denote by $O_g \subseteq T\R^{n+1}$ the domain of definition of the exponential map 
 with respect to the Riemannian metric $g$, and denote by $\exp: O_g \rightarrow \R^{n+1}$
 the associated exponential map. Denote the normal bundle of $\Sigma$ by $N \Sigma$.
 Then there is a normal exponential map $\exp^{\perp}: O_g \cap N \Sigma \rightarrow \R^{n+1}$, compare the book of J.~Lee, \cite[p.~144]{Lee-2018-RiemannianGeometry} or the book by P.~Petersen, \cite[p.~189--190]{Petersen-2016-RiemannianGeometry}.
 We define 
 \[
  U = \left\{ (t,x) \in [0,\infty) \times \Sigma: (x, -t \nu(x)) \in O_g \cap N\Sigma \right\} \subseteq [0,\infty) \times \Sigma
 \]
and the map $\Phi: U \rightarrow \R^{n+1}$ by $\Phi(t,x) = \exp^{\perp}(x, -t \nu(x))$.
It follows for all $x \in \Sigma$ that 
\begin{align*}
 \Phi(0,x) = x \quad \text{and} \quad 
 \left. \frac{\partial}{\partial t} \Phi(t,x) \right|_{t=0} = - \phi(x) \nu(x),
\end{align*}
and the map $t \mapsto \Phi(t,x)$ is a unit speed geodesic
for the metric $g$ for $(t,x) \in U$.

We define the \emph{segment domain of the hypersurface} $\Sigma$ by
\begin{align*}
 \seg(\Sigma) = \{(t,x) \in U: t = u(\Phi(t,x)) \}
\end{align*}
and the \emph{interior of the segment domain of the the hypersurface} $\Sigma$ by
\begin{align*}
 \seg^0(\Sigma) = \{(t,x) \in U: \exists \delta > 0: (t+\delta, x) \in \seg(\Sigma) \}.
\end{align*}

We have the following proposition, which is stated as Proposition 3.1 in work of S.~Brendle \cite{Brendle-2013-Alexandrov-WarpedProducts}:
 
 \begin{proposition}[Properties of segment domains for normal exponential map of hypersurfaces] \ \\
 The segment domains $\seg(\Sigma)$ and $\seg^0(\Sigma)$ have the following properties:
  \begin{enumerate}
   \item If $(t_0,x) \in \seg(\Sigma)$, then $(t,x) \in \seg(\Sigma)$ for all $t \in [0,t_0]$.
   \item The segment domain $\seg(\Sigma)$ is a closed set, and we have $\Phi(\seg(\Sigma))
   = \overline{\Omega}$.
   \item The interior segment domain $\seg^0(\Sigma)$ is an open subset of $U$
   and the restriction $\Phi|_{\seg^0(\Sigma)}$ is a diffeomorphism.
  \end{enumerate}
 \end{proposition}

 We then have as stated below Proposition 3.1 in work of S.~Brendle
 \cite{Brendle-2013-Alexandrov-WarpedProducts}:
 \begin{proposition} \ \\
  We define the maximal time $T := \sup\{ t \in [0,\infty): \exists x \in \Sigma: (t,x) \in U \}$. and for $t \in [0,T)$
  we define the subset $\Sigma_t \subseteq \R^{n+1}$ by $\Sigma_t = \Phi(\seg^0(\Sigma) \cap (\{t\} \times \Sigma))$. Then $\Sigma_t$ is a smooth hypersurface for all
  $t \in [0,T)$ and we have $\Sigma_0 = \Sigma$. 
 \end{proposition}

 We observe that the mean curvature stays positive along the evolution; this
 can be seen by combining the Jacobi equation for the mean curvature with the 
 trace inequality for the second fundamental form for the expander weight,
 Lemma \ref{Lemma:TraceInequality}, and the weighted Laplacian of the velocity, Lemma 
 \ref{Lemma:HypersurfaceIdentity}.
 
 \begin{proposition}[Evolution of weighted mean curvature] \ \\
  The weighted mean curvature $t \mapsto H_{\rho}^{\Sigma_t}$ is a positive smooth function on $[0,T)$. Morevoer, we have the differential inequality
  \begin{align*}
   \frac{\partial}{\partial t} H_{\rho}^{\Sigma_t}
   \geq (H_{\rho}^{\Sigma_t})^2 - 4 \alpha H_{\rho}^{\Sigma_t} \langle \Phi(t,x), \nu_t \rangle.
  \end{align*}
  \label{Proposition:Evolution-WeightedMeanCurvature}
 \end{proposition}

 \begin{proof} 
The mean curvature $H^{\Sigma}_t$ and the normal vector fields satisfy the evolution
equation 
\begin{align*}
 \frac{\partial}{\partial t} H^{\Sigma_t}
 = \Delta^{\Sigma_t} \phi + |h^{\Sigma_t}|^2 \phi
 \quad 
 \text{and}
 \quad 
 \frac{\partial}{\partial t} \nu^{\Sigma_t}
 = D^{\Sigma_t} \phi,
\end{align*}
compare for example \cite[Theorem 3.2]{Huisken-Polden-1999-EvolutionEquations}.
From the definition of the weighted mean curvature and the weighted Laplacian
we then deduce that 
  \begin{align*}
   \partial_t H_{\rho}
   &= \partial_t \left( H + 2 \alpha \langle x, \nu \rangle \right) \\
   &= \partial_t H + 2 \alpha \left( \langle \partial_t x, \nu \rangle + \langle x, \partial_t \nu \rangle \right) \\
   &= \Delta^{\Sigma} \phi + |h|^2 \phi + 2 \alpha \left( \langle -\phi \nu, \nu \rangle +
   \langle x, \nabla^{\Sigma} \phi \rangle \right) \\
   &= \Delta^{\Sigma} \phi + 2 \alpha \langle x, \nabla^{\Sigma} \phi \rangle + (|h|^2 - 2 \alpha) \phi \\
   &= \Delta^{\Sigma}_{\rho} \phi + (|h|^2 - 2 \alpha) \phi.
  \end{align*}
  
  The trace inequality for the second fundamental form for the expander weight, 
  Lemma \ref{Lemma:TraceInequality}, and Lemma \ref{Lemma:HypersurfaceIdentity} now imply that 
  \begin{align*}
   \partial_t H_{\rho}
   &= 4 \alpha (n + 2 \alpha |x|^2 - H_{\rho} \langle x, \nu \rangle) + (|h|^2 - 2 \alpha)(n + 2 \alpha |x|^2) \\
   &= (|h|^2 + 2 \alpha)(n + 2 \alpha |x|^2) - 4 \alpha H_{\rho} \langle x, \nu \rangle \\
   &\geq H_{\rho}^2 - 4 \alpha H_{\rho} \langle x, \nu \rangle.
  \end{align*}
  
 We deduce that 
\begin{align*}
 \frac{\partial}{\partial t}  \left( \frac{H_{\rho}^{\Sigma_t}}{\phi} \right)
 = 
 \frac{(\partial_t H_{\rho}^{\Sigma_t}) \phi - H_{\rho}^{\Sigma_t} \partial_t \phi}{\phi^2}
 =
 \frac{(H_{\rho}^{\Sigma_t})^2 \phi - 4 \alpha H_{\rho}^{\Sigma_t} \langle x, \nu \rangle \phi + H_{\rho}^{\Sigma_t} (4\alpha \langle x, - \phi \nu \rangle)}{\phi^2}
 = \frac{(H_{\rho}^{\Sigma_t})^2}{\phi} \geq 0.
\end{align*}

Since $\phi > 0$ on $\R^{n+1}$ the initial condition $H_{\rho}^{\Sigma} > 0$
implies that $H_{\rho}^{\Sigma_t} > 0$ for all $t \in [0,T)$.
This concludes the proof.
 \end{proof}

 \begin{proposition}[Evolution of area] \ \\
 The weighted area element $d\sigma_{\rho}^{\Sigma_t}$ satisties the differential inequality
 \begin{align*}
  \frac{\partial}{\partial t} d\sigma_{\rho}^{\Sigma_t} = - \phi H_{\rho}^{\Sigma_t} \, d\sigma_{\rho}^{\Sigma_t}.
 \end{align*}
 Moreover, the weighted area $t \mapsto A_{\rho}(\Sigma_t)$ is monotone decreasing.
 \label{Proposition:Evolution-WeightedArea}
 \end{proposition}

 \begin{proof}
 The volume element evolves by the equation
 \begin{align*}
  \frac{\partial}{\partial t} d\sigma^{\Sigma_t}
   = - \phi H \, d\sigma^{\Sigma_t},
 \end{align*}
 compare \cite[Theorem 3.2]{Huisken-Polden-1999-EvolutionEquations}.
 This implies for the weighted volume element
\begin{align*}
 \frac{\partial}{\partial t} d\sigma_{\rho}^{\Sigma_t}
 = - \phi H^{\Sigma_t} (\rho \; d\sigma^{\Sigma_t}) 
 + (\alpha \partial_t |x|^2) (\rho \; d\sigma^{\Sigma_t}) 
 = - \phi (H^{\Sigma_t} + 2 \alpha \langle x, \nu \rangle)  d\sigma_{\rho}^{\Sigma_t} 
 = - \phi H_{\rho}^{\Sigma_t} \; d\sigma_{\rho}^{\Sigma_t}.
\end{align*}
Since $H_{\rho}^{\Sigma_t} > 0$ by the previous proposition and $\phi > 0$ by definition
integration yields that the weighted area is monotone decreasing.
\end{proof}

For $t \in [0,T)$ we consider the Heintze--Karcher type quantity, compare \cite[p.~257]{Brendle-2013-Alexandrov-WarpedProducts},
\begin{align*}
 Q(t) = \int_{\Sigma_t} \frac{n + 2 \alpha |x|^2}{H_{\rho}^{\Sigma_t}} \; d\sigma_{\rho}^{\Sigma_t}.
\end{align*}
The combination of Proposition \ref{Proposition:Evolution-WeightedMeanCurvature} and \ref{Proposition:Evolution-WeightedArea} already shows that the above quantity
is non-increasing. However, we need the following refined estimate:

 \begin{proposition}[Evolution of Heintze--Karcher quantity] \ \\
  The Heintze--Karcher quantity $t \mapsto Q(t)$ satisfies for $t \in [0,T)$
  the inequality
  \begin{align*}
   Q(0) - Q(t) \geq \int_{d_g(\cdot, \Sigma) \leq t} \left(n+1+2\alpha |x|^2 \right) \, dx_{\rho}.
  \end{align*}
  \label{Proposition:Evolution-HeintzeKarcherInequality}
 \end{proposition}

 \begin{proof}
 The monotonicity of the function $t \mapsto Q(t)$
 and the evolution of the weighted mean curvature, Proposition \ref{Proposition:Evolution-WeightedMeanCurvature}, imply for $t \in (0,T)$ that 
 \begin{align*}
  &\limsup_{h \rightarrow 0} \frac{Q(t) - Q(t-h)}{h} \\
  \leq&
 \int_{\Sigma_t}
 \frac{\partial}{\partial t}
 \left(
  \frac{n + 2 \alpha |x|^2}{H_{\rho}}
 \right) 
 \; d\sigma_{\rho}
 +
 \int_{\Sigma_t} 
 \frac{n + 2 \alpha |x|^2}{H_{\rho}} \frac{\partial}{\partial t} d\sigma_{\rho} \\
 =&
 \int_{\Sigma_t} 
 \frac{2\alpha H_{\rho} \partial_t |x|^2 -(n+2\alpha |x|^2) \partial_t H_{\rho})}{H_{\rho}^2} \; d\sigma_{\rho} 
 +
 \int_{\Sigma_t}
 \frac{n + 2 \alpha |x|^2}{H_{\rho}} \left( - H_{\rho}(n + 2 \alpha |x|^2) \right) \; d\sigma_{\rho}
 \\
 \leq& 
 \int_{\Sigma_t}
 \frac{- 4 \alpha \phi H_{\rho} \langle x, \nu \rangle
 - (n + 2 \alpha |x|^2) (H_{\rho}^2 - 4 \alpha H_{\rho} \langle x, \nu \rangle)}{H_{\rho}^2} \; d\sigma_{\rho} 
  - \int_{\Sigma_t} (n + 2 \alpha |x|^2)^2 \; d\sigma_{\rho} \\
  =&
  - \int_{\Sigma_t} (n+1 + 2 \alpha |x|^2) \phi \; d\sigma_{\rho}.
\end{align*}

The Fundamental Theorem of Calculus for monotone functions
and the coarea formula then imply
\begin{align*}
Q(0) - Q(\tau)
 &\geq \int_{0}^{t} \left( \int_{\Sigma_t} (n+1 + 2\alpha |x|^2) \phi \; d\sigma_{\rho} \right) \; dt \\
 &= \int_{d_g(\cdot, \Sigma) \leq t} (n + 1 + 2\alpha |x|^2) \; dx_{\rho}.
\end{align*}
This proves the inequality.
 \end{proof}

 \begin{proof}[Proof of weighted Heintze--Karcher inequality, Theorem \ref{Theorem:HeintzeKarcherInequality}] \ \\
  Since $Q(t) > 0$ for all $t \in [0,T)$ the evolution of the weighted Heintze--Karcher quantity, Proposition \ref{Proposition:Evolution-HeintzeKarcherInequality}, implies that 
  \begin{align*}
   \int_{\Sigma} \frac{n + 2 \alpha |x|^2}{H_{\rho}^{\Sigma}} \; d\sigma_{\rho}
   = Q(0) \geq Q(0) - Q(\tau) 
   \geq \int_{d_g(\cdot, \Sigma) \leq t} (n + 1 + 2\alpha |x|^2) \; dx_{\rho}.
  \end{align*}
 Passing to the limit as $t \rightarrow T$ yields 
 \begin{align*}
  \int_{\Sigma} \frac{n + 2 \alpha |x|^2}{H_{\rho}^{\Sigma}}
  \geq \int_{\Omega} (n + 1 + 2\alpha |x|^2) \; dx_{\rho}.
 \end{align*}
 This shows the inequality.
 \end{proof}

 \begin{proof}[Proof of Alexandrov's Theorem, Theorem \ref{Theorem:AlexandrovTheorem-ExpanderWeight}] \ \\
    For $H_{\rho} = const > 0$ we combine the weighted Heintze--Karcher inequality,
    Theorem \ref{Theorem:HeintzeKarcherInequality} and the weighted Minkowski formula,
    Lemma \ref{Lemma:WeightedMinkowskiFormula}, to deduce
  \begin{align*}
   \int_{\Omega} (n+1 + 2 \alpha |x|^2) \; dx_{\rho}
   &\leq \int_{\Sigma} \frac{n + 2 \alpha |x|^2}{H_{\rho}} \; d\sigma_{\rho}
   = \frac{1}{H_{\rho}} \int_{\Sigma} (n + 2 \alpha |x|^2) \; d\sigma_{\rho} \\
   &= \frac{1}{H_{\rho}} \int_{\Sigma} H_{\rho} \langle x, \nu \rangle \; d\sigma_{\rho}
   = \int_{\Sigma} \langle x, \nu \rangle \; d\sigma_{\rho}
   = \int_{\Omega} (n+1 + 2 \alpha |x|^2) \; dx_{\rho}.
  \end{align*}
  This enforces equality in the weighted Heintze--Karcher inequality,
  Theorem \ref{Theorem:HeintzeKarcherInequality}. Therefore, we need
  to have equality in the geometric inequality for the expander weight, Lemma \ref{Lemma:TraceInequality}.
  Therefore $\Sigma$ is a sphere centered at the origin.
 \end{proof}

\appendix 

\section{Generalization to spaces of constant negative curvature and more general weights}
 
 In this appendix we discuss a generalization of the Alexandrov theorem to more general weights. We start by proving a generalization of the Minkowski identity
 for expander space, Lemma \ref{Lemma:WeightedMinkowskiFormula},
 to general weights in simply--connected spaces of constant nonpositive curvature.
 
 \begin{proposition}[Generalized weighted Minkowski formula] \ \\
  Let $n\geq 2$ and  $\kappa\leq 0$. Suppose $\Sigma^n$ is a closed and embedded submanifold of the weighted space $(M_\kappa^{n+1},\delta_\kappa,\rho\,dx)$. Then we have the Minkowski identity
\begin{equation}
    \int_{\Sigma} H_{\rho} \, \langle \sn_{\kappa}(r) \partial_r,\nu \rangle \,d\sigma_{\rho}
    = \int_{\Sigma} \left( n \cs_{\kappa}(r)+\sn_{\kappa}(r)
    \langle D\log \rho, \partial_r \rangle \right) \, d\sigma_{\rho}.
\end{equation}
  \label{Lemma:WeightedMinkowskiFormula-Generalized}
 \end{proposition}

\begin{proof}
 We define $\Phi(r) = \int_0^r \sn_{\kappa}(s) \; ds$, and denote by $\partial_r$ the radial vector field. Then we have the following well-known identity
 \begin{align*}
   \Delta^{\Sigma} \Phi
   = n \cs_{\kappa}(r) - H \sn_{\kappa}(r) \langle \partial_r, \nu \rangle.
 \end{align*}
For the convenience of the reader, we will provide a proof: 
Suppose $p \in \Sigma$ and $\{e_i\}_{i=1}^{n+1}$ is an orthonormal frame adapted
to $\Sigma$. By the Fundamental Theorem
of Calculus we deduce
\begin{align*}
 D_{e_i} \Phi(r) = \sn_{\kappa}(r) \langle \partial_r, e_i \rangle.
\end{align*}
Taking another derivative yields
\begin{align*}
 D_{e_i} (D_{e_i} \Phi(r))
 &= D_{e_i} (\sn_{\kappa}(r)) \langle \partial_r, e_i \rangle
 + \sn_{\kappa}(r) \langle D_{e_i} \partial_r, e_i \rangle
 + \sn_{\kappa}(r) \langle \partial_r, D_{e_i} e_i \rangle \\
 &= \cs_{\kappa}(r) \langle \partial_r, e_i \rangle^2
 + \sn_{\kappa}(r) \left \langle
 \frac{\cs_{\kappa}(r)}{\sn_{\kappa}(r)} (e_i - \langle e_i, \partial_r \rangle \partial_r)
 , e_i \right \rangle
 - \sn_{\kappa}(r) h(e_i, e_i) \langle \partial_r, \nu \rangle \\
 &= \cs_{\kappa}(r) - \sn_{\kappa}(r) h(e_i, e_i) \langle \partial_r, \nu \rangle,
\end{align*}
where we have used the identity
\begin{align*}
 \nabla_X \partial_r 
 = \frac{\cs_{\kappa}(r)}{\sn_{\kappa}(r)} (X - \langle \partial_r, X \rangle X)
\end{align*}
for the covariant derivative of the radial vector field. Summations over $i \in \{1, \dots, n\}$
yields the identity.

Combining the above identity with the definition of the weighted Laplacian yields
\begin{align*}
 \Delta^{\Sigma}_{\rho} \Phi
 &=
 \Delta^{\Sigma} \Phi
 + \langle D^{\Sigma} \log \rho, D^{\Sigma} \Phi \rangle \\
 &=
 n \cs_{\kappa}(r) - \sn_{\kappa}(r) H \langle \partial_r, \nu \rangle 
 +
 \sum_{i=1}^n D_{e_i} \log \rho \, \sn_{\kappa}(r) \langle \partial_r, e_i \rangle \\
 &=
n \cs_{\kappa}(r) - \sn_{\kappa}(r) H \langle \partial_r, \nu \rangle 
+ \sn_{\kappa}(r) \langle D \log \rho, \partial_r \rangle 
- \sn_{\kappa}(r) \langle D \log \rho, \nu \rangle \langle \partial_r, \nu \rangle \\
&=
 n \cs_{\kappa}(r) + \sn_{\kappa}(r) \langle D \log \rho, \partial_r \rangle 
 - \sn_{\kappa}(r) H_{\rho} \langle \partial_r, \nu \rangle.
\end{align*}
Multiplication by $\rho$, and the divergence theorem implies the desired identity. 
\end{proof}

 The next ingredient is a generalization of the trace inequality
 for the second fundamental for the expander weight,
 Lemma \ref{Lemma:TraceInequality}, which relates the second fundamental form
 to the weighted mean curvature.
 
 \begin{lemma}[Generalized trace inequality] \ \\
  Let $n\geq 2$. Suppose $F: \Sigma^n \rightarrow M^{n+1}$ is an immersion
  of the closed $n$-manifold $\Sigma^n$ into the weighted space
  $(M_\kappa^{n+1},\delta_\kappa,\exp(\psi(r(x)))\,dx)$. 
  We define the radial projection
  $\nu^{\perp} = \langle \nu, \partial_r \rangle \, \partial_r$
  and its complement $\nu^{\parallel} = \nu - \nu^{\perp}$.
 Then we have the following pointwise identity  
    \begin{equation}
        \begin{aligned}
        &\left(
        |h|^2+ \psi''(r) |\nu^{\parallel}|^2 + \psi'(r)\frac{\cs_{\kappa}(r)}{\sn_{\kappa}(r)}|\nu^{\perp}|^2
        \right)
        \left( n\cs_{\kappa}(r)+ \psi'(r) \sn_{\kappa}(r) \right) - \cs_{\kappa}(r) H_{\rho}^2\\
        =& |\overset{\circ}{h}|^2 \left( n\cs_{\kappa}(r) + \psi'(r)\sn_{\kappa}(r) \right)
        + \psi'(r) \sn_{\kappa}(r) \frac{H^2}{n} |\nu^{\parallel}|^2
        +\frac{\psi'(r)}{n\sn_{\kappa}(r)}
        \left(
        \sn_{\kappa}(r) H \langle \nu, \partial_r \rangle - n\cs_{\kappa}(r) 
        \right)^2\\
        +&\left(
        \psi''(r) \left(n\cs_{\kappa}(r) + \psi'(r)\sn_{\kappa}(r) \right)
        - n\psi'(r)\frac{\cs_{\kappa}(r)^2}{\sn_{\kappa}(r)}
        \right)
        |\nu^{\parallel}|^2.
        \end{aligned}
    \end{equation}
    Moreover, the right-hand side is equal to zero on a geodesic sphere centered at $o$.
    \label{Lemma:TraceInequality-Generalized}
 \end{lemma}

\begin{proof}
We decompose the second fundamental form into its tracefree part and the mean curvature.
   \begin{align*}
        &\left(
        |h|^2+\psi''(r)|\nu^{\parallel}|^2 + \psi'(r) \frac{\cs_{\kappa}(r)}{\sn_{\kappa}(r)}|\nu^{\perp}|^2
        \right)
        \left(
        n\cs_{\kappa}(r)+\psi'(r)\sn_{\kappa}(r)
        \right)
        -\cs_{\kappa}(r) H_{\rho}^2\\
        =& |\overset{\circ}{h}|^2
        \left( n\cs_{\kappa}(r)+\psi'(r)\sn_{\kappa}(r) \right)
        +\cs_{\kappa}(r) H^2 + \psi'(r)\sn_{\kappa}(r)\frac{H^2}{n} - \cs_{\kappa}(r) H_{\rho}^2\\
        +&\left(
        \psi''(r)|\nu^{\parallel}|^2+\psi'(r) \frac{\cs_{\kappa}(r)}{\sn_{\kappa}(r)}|\nu^{\perp}|^2
        \right)
        \left(
        n\cs_{\kappa}(r)+\psi'(r)\sn_{\kappa}(r)
        \right)\\
        =& |\overset{\circ}{h}|^2
        \left( n\cs_{\kappa}(r)+\psi'(r)\sn_{\kappa}(r) \right)
        +\cs_{\kappa}(r) H^2 + \psi'(r)\sn_{\kappa}(r)\frac{H^2}{n}(|\nu^\perp|^2+|\nu^\parallel|^2)\\
        -&\cs_{\kappa}(r) (H^2+2H\psi'(r)\langle\partial_r,\nu\rangle+\psi'(r)^2\langle\partial_r,\nu\rangle^2)\\
        +&\left(
        \psi''(r)|\nu^{\parallel}|^2+\psi'(r) \frac{\cs_{\kappa}(r)}{\sn_{\kappa}(r)}|\nu^{\perp}|^2
        \right)
        \left(
        n\cs_{\kappa}(r)+\psi'(r)\sn_{\kappa}(r)
        \right)\\
         =& |\overset{\circ}{h}|^2
        \left( n\cs_{\kappa}(r)+\psi'(r)\sn_{\kappa}(r) \right)
       + \psi'(r)\sn_{\kappa}(r)\frac{H^2}{n}|\nu^\parallel|^2\\
       +&\frac{\psi'(r)}{n\sn_\kappa(r)}(\sn_\kappa(r)H\langle\nu,\partial_r\rangle-n\cs_\kappa(r))^2-n\psi'(r)\frac{\cs_\kappa(r)^2}{\sn_\kappa(r)}\\
        -&\cs_{\kappa}(r)\psi'(r)^2\langle\partial_r,\nu\rangle^2\\
        +&\left(
        \psi''(r)|\nu^{\parallel}|^2+\psi'(r) \frac{\cs_{\kappa}(r)}{\sn_{\kappa}(r)}|\nu^{\perp}|^2
        \right)
        \left(
        n\cs_{\kappa}(r)+\psi'(r)\sn_{\kappa}(r)
        \right)\\
        =&|\overset{\circ}{h}|^2
        \left(
        n\cs_{\kappa}(r)+\psi'(r)\sn_{\kappa}(r)
        \right)
        +\psi'(r) \sn_{\kappa}(r) \frac{H^2}{n}|\nu^{\parallel}|^2 \\
        +&\frac{\psi'(r)}{n\sn_{\kappa}(r)}
        \left(
        \sn_{\kappa}(r) H \langle \nu, \partial_r \rangle - n\cs_{\kappa}(r)
        \right)^2
        +\left(
        \psi''(r) (n\cs_{\kappa}(r)+ \psi'(r)\sn_{\kappa}(r)) - n\psi'(r)\frac{\cs_{\kappa}(r)^2}{\sn_{\kappa}(r)}
        \right)|\nu^{\parallel}|^2.
    \end{align*}
This is the desired decomposition.
\end{proof}

In the proof of Alexandrov's Theorem for expander space, Theorem \ref{Theorem:AlexandrovTheorem-ExpanderWeight} the third ingredient was a generalization of the classical Heintze--Karcher inequality. This inequality holds for a more general weights defined by 
a differential inequality. 
 
 \begin{theorem}[Generalized Heintze--Karcher inequality] \ \\
  Let $n \geq 2$, $\kappa \leq 0$ and let $\rho: M^{n+1}_{\kappa} \rightarrow (0,\infty)$
 be a nonconstant positive radially symmetric weight, i.e.\ there exists $\psi: (0,\infty) \rightarrow \R$, such that $\rho(x) = \exp(\psi(r(x)))$, and $\psi$ additionally satisfies
 for all $r >0$ the inequalities $\psi'(r) \geq 0$ and 
 \begin{equation}
 A_{\kappa}(r) =  \frac{d}{dr} \left( \psi''(r) + \frac{(\psi'(r))^2}{2} + (n+1) \frac{\cs_{\kappa}(r) \psi'(r)}{\sn_{\kappa}(r)} + \kappa n \psi(r) 
  \right) \geq 0.
  \label{equation:Monotonicity}
 \end{equation}
 Suppose $\Sigma^n$ is a closed, connected and embedded submanifold in $(M^{n+1}_{\kappa}, \delta_k, \rho \, dx)$ with positive weighted mean curvature $H_{\rho} > 0$ bounding a non-empty open and connected
 subset $\Omega$ of $M^{n+1}_{\kappa}$.
Then we have the weighted Heintze--Karcher inequality
    \begin{equation}
        \int_{\Omega}
        \left( (n+1)\cs_{\kappa}(r) +\psi'(r) \sn_{\kappa}(r) \right) \,dx_{\rho}
        \leq \int_{\Sigma} \frac{n \cs_{\kappa}(r) +\psi'(r)\sn_{\kappa}(r)}{H_{\rho}} \, d\sigma_{\rho}.
    \end{equation}
  Moreover, if equality holds then $\Sigma$ is a round sphere centered at $o \in M^{n+1}_{\kappa}$.  
  \label{Theorem:HeintzeKarcherInequality-Generalized}
\end{theorem}

 \begin{proof}[Sketch of the proof of Theorem \ref{Theorem:HeintzeKarcherInequality-Generalized}]
 The proof follows the same strategy as the proof for the expander weight.
 We define $\phi: M^{n+1}_{\kappa} \rightarrow (0,\infty)$ by
 \[
   \phi(x) =  n\cs_{\kappa} (r(x)) + \psi'(r(x)) \sn_{\kappa} (r(x)),
 \]
 and consider the normal geodesic flow $(\Sigma_t)_{t \in [0,T)}$ with respect to the conformal metric $g:=\phi^{-2}\delta_\kappa$ starting at $\Sigma_0=\Sigma$. Then
 we compute the gradient
\begin{equation*}
   D\phi = \left(\partial_r \phi(x) \right) \partial_r
   =\left( -n \kappa \sn_{\kappa}(r) +\psi'(r)\cs_{\kappa}(r)+ \psi''(r)\sn_{\kappa}(r) \right) \partial_r,
 \end{equation*}
 and the Hessian
 
 \begin{align*}
D^2\phi
=& \partial_r^2\phi(x) \partial_r \otimes \partial_r + \partial_r \phi(x) \frac{\cs_{\kappa}(r)}{\sn_{\kappa}(r)} (\id - \partial_r \otimes \partial_r)\\
=& \left(-n \kappa \cs_{\kappa}(r) - \kappa \psi'(r) \sn_{\kappa}(r) + 2\psi''(r) \cs_{\kappa}(r) + \psi'''(r)\sn_{\kappa}(r) \right) \left( \partial_r\otimes\partial_r \right)\\
&+\left(-n \kappa \cs(r) + \psi'(r) \frac{\cs_{\kappa}(r)^2}{\sn_{\kappa}(r)} + \psi''(r) \cs_{\kappa}(r)\right) (\id - \partial_r \otimes \partial_r)\\
=& -n\kappa\cs_\kappa(r)\id+\left(- \kappa \psi'(r) \sn_{\kappa}(r) + \psi''(r) \cs_{\kappa}(r) + \psi'''(r)\sn_{\kappa}(r) \right) \left( \partial_r\otimes\partial_r \right)\\
&+\cs_\kappa(r)\left( \psi''(r) \partial_r \otimes \partial_r + \psi'(r) \frac{\cs_{\kappa}(r)}{\sn_{\kappa}(r)} (\id - \partial_r \otimes \partial_r)\right)\\
&+\psi''(r) \cs_{\kappa}(r) (\id - \partial_r \otimes \partial_r)\\
=&(\psi''(r) - n\kappa) \cs_{\kappa}(r) \id + \cs_{\kappa}(r) D^2 \log \rho 
+ (\psi'''(r)- \kappa \psi'(r)) \sn_{\kappa}(r) (\partial_r \otimes \partial_r).
\end{align*}

 This implies for the weighted Laplace operator on the hypersurface $\Sigma_t$ the expression
\begin{align*}
\Delta_{\rho}^{\Sigma_t}\phi
=& \Delta^{M^{n+1}_{\kappa}} \phi - D^2\phi(\nu,\nu) -H_{\rho}^{\Sigma_t} \langle D\phi, \nu \rangle + \psi'(r) \langle \partial_r,  D\phi \rangle \\
=&(n+1) (\psi''(r) - n\kappa) \cs_{\kappa}(r) + \cs_{\kappa}(r) \Delta^{M^{n+1}_{\kappa}} \log \rho 
+ (\psi'''(r) - \kappa \psi'(r)) \sn_{\kappa}(r)\\
&- (\psi''(r)-n\kappa) \cs_{\kappa}(r) - \cs_{\kappa}(r) D^2 \log \rho(\nu,\nu)
- (\psi'''(r) - \kappa \psi'(r)) \sn_{\kappa}(r) |\nu^{\perp}|^2 \\
&+ \psi'(r) (-n \kappa \sn_{\kappa}(r) + \psi'(r) \cs_{\kappa}(r) +\psi''(r) \sn_{\kappa}(r)) 
- H_{\rho}^{\Sigma_t} \langle D\phi, \nu \rangle \\
=& (\psi''(r) - n\kappa) \phi +\cs_{\kappa}(r) \Delta^{M^{n+1}_{\kappa}} \log \rho
+ (\psi'''(r) - \kappa\psi'(r)) \sn_{\kappa}(r) |\nu^{\parallel}|^2 + (\psi'(r))^2\cs_{\kappa}(r)\\
&- \cs_{\kappa}(r) D^2 \log \rho(\nu,\nu) - H_{\rho}^{\Sigma_t} \langle D\phi, \nu \rangle.
 \end{align*}

We deduce for the evolution of the weighted mean curvature that
\begin{align*}
\partial_t H_{\rho}^{\Sigma_t}
=&\Delta^{\Sigma_t}_{\rho} \phi + (|h|^2+ n\kappa - D^2 \log \rho(\nu,\nu)) \phi \\
=&(|h|^2 + \psi''(r) - D^2 \log \rho(\nu,\nu)) \phi \\
&+\cs_{\kappa}(r) \Delta^{M^{n+1}_{\kappa}} \log \rho + (\psi'''(r)-\kappa\psi'(r)) \sn_{\kappa}(r) |\nu^{\parallel}|^2
+ (\psi'(r))^2 \cs_{\kappa}(r)\\
&-\cs_{\kappa}(r) D^2 \log \rho(\nu,\nu) - H_{\rho}^{\Sigma_t} \langle D\phi, \nu \rangle \\
=&\left(|h|^2 + \psi''(r)(|\nu^\perp|^2+|\nu^\parallel|^2) - \psi''(r)|\nu^\perp|^2-\frac{\cs_\kappa(r)}{\sn_\kappa(r)}|\nu^\parallel|^2\psi'(r)\right) \phi \\
&+\cs_{\kappa}(r) \left(\psi''(r)+n\frac{\cs_\kappa(r)}{\sn_\kappa(r)}\psi'(r)\right) + (\psi'''(r)-\kappa\psi'(r)) \sn_{\kappa}(r) |\nu^{\parallel}|^2
+ (\psi'(r))^2 \cs_{\kappa}(r)\\
&-\cs_{\kappa}(r) \left(\psi''(r)|\nu^\perp|^2+\frac{\cs_\kappa(r)}{\sn_\kappa(r)}|\nu^\parallel|^2\psi'(r)\right)- H_{\rho}^{\Sigma_t} \langle D\phi, \nu \rangle.
\end{align*}
Rearranging the terms yields
\begin{align*}
\partial_t H_{\rho}^{\Sigma_t}
=&\left(|h|^2 + \psi''(r)|\nu^\parallel|^2-\frac{\cs_\kappa(r)}{\sn_\kappa(r)}|\nu^\parallel|^2\psi'(r)\right) \phi \\
&+\cs_\kappa(r)\psi''(r)|\nu^\parallel|^2+n\frac{\cs^2_\kappa(r)}{\sn_\kappa(r)}\psi'(r) + (\psi'''(r)-\kappa\psi'(r)) \sn_{\kappa}(r) |\nu^{\parallel}|^2
+ (\psi'(r))^2 \cs_{\kappa}(r)\\
&-\frac{\cs_\kappa^2(r)}{\sn_\kappa(r)}|\nu^\parallel|^2\psi'(r)- H_{\rho}^{\Sigma_t} \langle D\phi, \nu \rangle \\
=&\left(|h|^2 + \psi''(r)|\nu^\parallel|^2-\frac{\cs_\kappa(r)}{\sn_\kappa(r)}|\nu^\parallel|^2\psi'(r)\right) \phi \\
&+\cs_\kappa(r)\psi''(r)|\nu^\parallel|^2+\frac{\cs_\kappa(r)}{\sn_\kappa(r)}\psi'(r)(n\cs_\kappa(r)+\psi'(r)\sn_\kappa(r)) + (\psi'''(r)-\kappa\psi'(r)) \sn_{\kappa}(r) |\nu^{\parallel}|^2
\\
&-\frac{\cs_\kappa^2(r)}{\sn_\kappa(r)}|\nu^\parallel|^2\psi'(r)- H_{\rho}^{\Sigma_t} \langle D\phi, \nu \rangle \\
=&\left(
|h|^2 + \psi''(r)|\nu^{\parallel}|^2 + \psi'(r) \frac{\cs_{\kappa}(r)}{\sn_{\kappa}(r)} |\nu^{\perp}|^2
\right)\phi\\
&+\cs_{\kappa}(r) \psi''(r)|\nu^{\parallel}|^2 + (\psi'''(r) - \kappa \psi'(r)) \sn_{\kappa}(r) |\nu^{\parallel}|^2
-\frac{\cs_{\kappa}(r)^2}{\sn_{\kappa}(r)} \psi'(r)|\nu^{\parallel}|^2
- H_{\rho}^{\Sigma_t} \langle D\phi, \nu \rangle.
\end{align*}

Hence, the generalized trace inequality Lemma \ref{Lemma:TraceInequality-Generalized}, the monotonicity of the weight $\psi$, and the assumption $A_{\kappa} \geq 0$
given by equation \eqref{equation:Monotonicity} imply that
\begin{align*}
    \partial_t H_{\rho}^{\Sigma_t}
    \geq& \cs_{\kappa}(r)(H_{\rho}^{\Sigma_t})^2 
    + \left(
    \psi''(r)(n\cs_{\kappa}(r)+\psi'(r)\sn_{\kappa}(r))-n\psi'(r)\frac{\cs_{\kappa}(r)^2}{\sn_{\kappa}(r)}
    \right)|\nu^{\parallel}|^2\\
    &+ \cs_{\kappa}(r) \psi''(r) |\nu^{\parallel}|^2
    + (\psi'''(r)- \kappa \psi'(r))\sn_{\kappa}(r)|\nu^{\parallel}|^2
    - \frac{\cs_{\kappa}(r)^2}{\sn_{\kappa}(r)} \psi'(r) |\nu^{\parallel}|^2
    - H_{\rho}^{\Sigma_t} \langle D\phi, \nu \rangle \\
    &= \cs_{\kappa}(r)(H_{\rho}^{\Sigma_t})^2 
    + \left(
    \psi''(r)(n+1)\cs_{\kappa}(r)+\psi'(r)\psi''(r)\sn_\kappa(r)-(n+1)\psi'(r)\frac{\cs_{\kappa}(r)^2}{\sn_{\kappa}(r)}
    \right)|\nu^{\parallel}|^2\\
    &
    + (\psi'''(r)- \kappa \psi'(r))\sn_{\kappa}(r)|\nu^{\parallel}|^2
    - H_{\rho}^{\Sigma_t} \langle D\phi, \nu \rangle.
   \end{align*}
   We collect terms proportional to $\nu^{\parallel}$:
   \begin{align*}
    \partial_t H_{\rho}^{\Sigma_t}
    \geq& 
    \cs_{\kappa}(r)(H_{\rho}^{\Sigma_t})^2
    - H_{\rho}^{\Sigma_t} \langle D\phi, \nu \rangle \\
    &+\left(
    (n+1)\cs_{\kappa}(r)\sn_{\kappa}(r)
    \left(
    \frac{\psi'(r)}{\sn_{\kappa}(r)}\right)'
    -\kappa \psi'(r)\sn_{\kappa}(r) + (\psi''(r)\psi'(r)+\psi'''(r)) \sn_{\kappa}(r)
    \right)|\nu^{\parallel}|^2\\
    &=\cs_{\kappa}(r) (H_{\rho}^{\Sigma_t})^2
    - H_{\rho}^{\Sigma_t} \langle D\phi, \nu \rangle \\
    &+\sn_{\kappa}(r)
    \left(
    \psi''(r) + \frac{(\psi'(r))^2}{2} + (n+1) \frac{\cs_{\kappa}(r)}{\sn(r)} \psi'(r)
    +\kappa n\psi(r)\right)'\\
    &\geq \cs_{\kappa}(r) (H_{\rho}^{\Sigma_t})^2
    -H_{\rho}^{\Sigma_t} \langle D\phi, \nu \rangle.
\end{align*}

Note that in the application of the generalized trace inequality Lemma \ref{Lemma:TraceInequality-Generalized} we discared the first three nonnegative terms on the right-hand side.

Once we obtain the above inequality, the argument works as in the Euclidean case.
We define the Heintze--Karcher type inequality
\begin{align*}
    Q(t) = \int_{\Sigma_t} \frac{n\cs_{\kappa}(r) + \psi'(r) \sn_{\kappa} (r)}
    {H_{\rho}^{\Sigma_t}}\,d\sigma_{\rho}^{\Sigma_t}.
\end{align*}
Then one can check that 
\begin{align*}
    \limsup_{h\to 0} \frac{Q(t)-Q(t-h)}{h}\leq-\int_{\Sigma_t}((n+1)\cs(r)+\psi'(r)\sn(r))\phi\ \; d\sigma_{\rho}^{\Sigma_t}.
\end{align*}

The conclusion follows by integration in the variable $t \in [0,T)$
followed by an application of the coarea formula. 
\end{proof}

 \section{Log--convexity of the weight and the differential inequality}
 
 In this appendix we show that $A_0 \geq 0$ implies log-convexity of the weight 
 
 We define the auxilliary function $B_0: (0,\infty) \rightarrow \R$ by 
 \begin{align*}
  B(r) = \psi''(r) + \frac{(\psi'(r))^2}{2} + (n+1) \frac{\psi'(r)}{r}.
 \end{align*}
 
 Observe that 
 \begin{align*}
   \frac{d}{dr} \left( r^{n+1} \exp \left( \frac{\psi(r)}{2} \right) \psi'(r) \right) = r^{n+1} \exp \left(\frac{\psi(r)}{2} \right) B(r).
 \end{align*}

 Since $A_0(r) = B'(r) \geq 0$ by assumption we may use the monotonicity
 assumption $\psi'(r) \geq 0$ and the regularity assumption $\psi'(0) = 0$
 to deduce
 \begin{align*}
  r^{n+1} \exp\left(\frac{\psi}{2} \right) \psi'(r)
  &=
  \int_0^r s^{n+1} \exp \left( \frac{\psi(r)}{2} \right) B(s) \, ds \\
  &\leq B(r) \int_0^r \exp \left( \frac{\psi(r)}{2} \right) s^{n+1} \, ds 
  \leq \frac{r}{n+2} r^{n+1} \exp \left( \frac{\psi(r)}{2} \right) B(r).
 \end{align*}
 This implies that $(n+2) \frac{\psi'(r)}{r} \leq B(r)$ or equivalently
 \begin{align*}
  (n+2) \frac{\psi'(r)}{r} \leq \psi''(r) + \frac{(\psi'(r))^2}{2} + (n+1) \frac{\psi'(r)}{r}.
 \end{align*}
 Rearranging yields
 \begin{align*}
  \psi''(r) \geq \frac{\psi'(r)}{r} - \frac{(\psi'(r))^2}{2}.
 \end{align*}
  Thus we have $\psi''(r) \geq 0$ whenever $\psi'(r) \leq \frac{2}{r}$, in particular this holds
  for small values of $r \in (0,\infty)$.
  
  We now proceed by contradiction. Assume there exists $r_0 \geq 0$, such that $\psi''(r_0) < 0$.
  Let $(a,b) \subseteq [0, +\infty]$ be the maximal interval, such that $r_0 \in (a,b)$
  and $\psi''|_{(a,b)} < 0$. By the previous observation we have $a > 0$, $b \in (a, +\infty]$
  and $\psi''(a) = 0$ by maximality.
  
  We define another auxilliary function $D(r) = r^{-(n+1)} \frac{d}{dr} (r^{n+1} \psi'(r))$;
  hence 
  \[
   B(r) = D(r) + \frac{(\psi'(r))^2}{2}.
  \]
 Differentiation yields for any $r \in (a,b)$ that 
 \[
  D'(r) = B'(r) - \psi'(r) \psi''(r) > 0,
 \]
 since $A_0 = B'(r) \geq 0$ and $\psi'(r) > 0$ while $\psi''(r) < 0$ on $(a,b)$.

 For $s \in (a,b)$ we obtain by the Fundamental Theorem of Calculus,
 integration by parts, and monotonicity of $r \mapsto D(r)$ in $(a,b)$ that 
 \begin{align*}
  s^{n+1} \psi'(s) - a^{n+1} \psi'(a)
  &=
  \int_a^s \xi^{n+1} D(\xi) \; d\xi
  =
  \frac{s^{n+2}}{n+2} D(s) - \frac{a^{n+2}}{n+2} D(a)
  - \int_a^s \frac{\xi^{n+2}}{n+2} D'(\xi) \, d\xi \\
  &< \frac{s^{n+2}}{n+2} D(s) - \frac{a^{n+2}}{n+2} D(a) \\
  &=  \frac{s^{n+2}}{n+2} \psi''(s)
  +
  s^{n+1} \frac{n+1}{n+2} \psi'(s)
  - 
  a^{n+1} \frac{n+1}{n+2} \psi'(a).
 \end{align*}

We rearrange the above identity to obtain
\begin{equation}
 s^{n+1} \psi'(s) - a^{n+1} \psi'(a) < s^{n+2} \psi''(s) < 0.
 \label{equation:Auxilliary}
\end{equation}
Dividing by $s > 0$ and sending $s \rightarrow 0$ yields
\begin{align*}
 \left. \frac{d}{dr} (r^{n+1} \psi'(r)) \right|_{r = a} \leq 0.
\end{align*}
By maximality we have $\psi''(a) = 0$ and hence
\begin{align*}
 0 \geq a^{n+1} \psi''(a) + (n+1) a^n \psi'(a) = (n+1) a^n \psi'(a).
\end{align*}
Since $\psi' \geq 0$ we deduce $\psi'(a) = 0$. Then equation \eqref{equation:Auxilliary}
yields a contradiction.

\end{document}